\documentclass[a4paper,twoside,reqno]{amsart} 

\usepackage[T1]{fontenc}
\usepackage[english]{babel}
\usepackage{microtype}

\usepackage{enumitem}
\setlist[enumerate]{label=\textit{(\roman*)},ref=(\roman*),align=left, leftmargin=*}

\usepackage{amsthm} 
\newtheorem{definition}{Definition}[section]
\newtheorem{theorem}[definition]{Theorem}
\newtheorem{proposition}[definition]{Proposition}

\newtheorem{lemma}[definition]{Lemma}
\newtheorem{remark}[definition]{Remark}
\newtheorem{example}[definition]{Example}

\usepackage{amsmath,amssymb,mathrsfs,mathtools,braket,textcomp,bm} 
\numberwithin{equation}{section}

\newcommand{\eps}{\varepsilon}
\newcommand{\ep}{\varepsilon}
\renewcommand{\epsilon}{\eps}
\renewcommand{\phi}{\varphi}
\newcommand{\N}{\mathbb{N}}
\newcommand{\R}{\mathbb{R}}
\newcommand{\Rd}{\R^d}
\newcommand{\de}{\mathrm{d}}
\newcommand{\vass}[1]{\left\lvert#1\right\rvert}
\newcommand{\norm}[1]{\left\lVert#1\right\rVert}

\newcommand{\id}{\mathrm{id}}
\DeclareMathOperator{\Per}{Per}
\newcommand{\PerK}{\Per_K}

\usepackage{comment}
\usepackage{xcolor}
\usepackage{hyperref}

\titlepage
\title{Convergence of nonlocal geometric flows to~anisotropic~mean~curvature~motion}

\author[A. Cesaroni]{Annalisa Cesaroni}
\address[A. Cesaroni]{
	Department of Statistical Sciences,
	Universit\`a di Padova,
	Via Battisti 241/243, 35121 Padova, Italy.
	Email: \href{mailto:annalisa.cesaroni@unipd.it}{\tt annalisa.cesaroni@unipd.it}.
}

\author[V. Pagliari]{Valerio Pagliari}
\address[V. Pagliari]{
	Institute of Analysis and Scientific Computing,
	TU Wien,
	Wiedner Hauptstrasse 8-10, 1040 Vienna, Austria.
	Email: \href{mailto:valerio.pagliari@tuwien.ac.at}{\tt valerio.pagliari@tuwien.ac.at}.
}

\date{\today}

\begin{document}

\maketitle
\begin{abstract}
	We consider  nonlocal curvature functionals
	associated with positive interaction kernels, and we show that
	local anisotropic mean curvature functionals can be retrieved
	in a blow-up limit from them.
	As a consequence, we  prove that
	the  viscosity solutions to the rescaled  nonlocal  geometric flows  
	locally uniformly converge to the viscosity solution
	to the anisotropic mean curvature motion.
	The result is achieved by combining
	a compactness argument and a set-theoretic approach
	related to the theory of De Giorgi's barriers for evolution equations.
	
	\medskip
	
	
\end{abstract}

\section{Introduction} 
In this paper we prove convergence
of a class of rescaled nonlocal curvature flows
to local anisotropic  mean curvature evolutions. 
 
We fix an interaction kernel $K\colon \Rd\setminus \{0\} \to [0,+\infty)$,
possibly singular at $0$, 
modeling interactions between points in the space,
and we define the \emph{nonlocal curvature} associated with $K$
of a measurable set $E\subseteq \Rd$ at $x\in\partial E$ as
	\begin{equation}\label{eq:dfncurv}
	H_K(E,x) :=-\lim_{r\to 0^+}\int_{B(x,r)^c}K(y-x)\tilde\chi_E(y)\de y.
	\end{equation}
Here and in the sequel, $B(x,r)$ is the open ball with center $x$ and radius $r$,
$E^c=\Rd\setminus E$ for any $E\subseteq \Rd$,
and $\tilde\chi_E(x)$ is equal to $1$ when $x\in E$
and it is equal to $-1$ otherwise.

Note that if $K\in L^1(\Rd)$,
then the nonlocal curvature coincides with $
	H_K(E,x) = -(K\ast \tilde\chi_E)(x)$. 
More generally,  we will  impose  conditions on $K$
so that $C^{1,1}$ sets have bounded nonlocal curvature, see  Section 2. 

By using the nonlocal curvature operator,
we define a nonlocal flow as follows:
for a family of evolving sets $\{E(t)\}_{t\geq 0}$,
we prescribe the geometric law
\begin{equation}\label{geometricflow}
\partial_t x(t)\cdot \hat{n} = -H_K(E(t), x),
\end{equation} 
where $\hat{n}$ is the outer unit normal 
to $\partial E(t)$ at the point $x(t)$. 

Geometric nonlocal evolutions as \eqref{geometricflow} emerged  as models for dislocations dynamics in the description of  plastic behavior of metallic crystals. 
Dislocations are  linear misalignments in the microscopic crystalline lattice, and
whose normal velocity is determined
by the so called Peach-Koehler force.
In \cite{ahl:dislocationdynamics},
Alvarez, Hoch, Le Bouar, and Monneau proposed a mathematical description
of dislocation dynamics in terms of a nonlocal eikonal equation,
where  the Peach-Koelher force is encoded by a convolution kernel $c_0$. 
The explicit expression of the kernel might be complicated,
because it has to capture the physical features of the system, e.g. in general it can change sign. 
By then, their model has been  simplified in a series of papers,
in which well-posedness of the geometric evolution law was obtained, 
see  \cite{acm:existenceand,bl:nonlocalfirst,imr:homogenizationof,fim:homogenizationof,dfm:convergenceof}.
 
Another interesting aspect of the nonlocal curvature \eqref{eq:dfncurv} is that
it is the first variation of the nonlocal perimeter functional
	\[\PerK(E):= \int_{E}\int_{E^c}K(y-x)\de y\de x\]
(see e.g. \cite{cmp:nonlocalcurvature}),
and the geometric evolution law \eqref{geometricflow}
is then understood 
as the $L^2$  gradient flow of this kind of perimeter. 
 
When $K$ belongs to an appropriate class of fractional kernels,
existence and uniqueness of solutions in the viscosity sense
to the geometric flow \eqref{geometricflow} were investigated in \cite{imb:levelset}.
More recently, 
Chambolle, Morini, and Ponsiglione have proved  in \cite{cmp:nonlocalcurvature}
well posedness of 
the level-set formulation of a wide class
of local and nonlocal translation-invariant geometric flows.
They also have exploited the minimizing movement scheme
to construct solutions to flows
driven by variational curvatures.

The analysis of nonlocal curvature flows as \eqref{geometricflow}
has lately been carried out from various perspectives,
especially in fractional case;
for instance, conservation of convexity, formation of neckpinch singularities,
and fattening phenomena
have been considered,
see \cite{cdn:fatteningand, csv,cnr:convex}.

As we anticipated, we are interested in the asymptotic behaviour
of a family of  nonlocal curvature flows, 
obtained by rescaling the kernel $K$.
Explicitly, for any $\ep>0$ and $x\in\Rd$,
we put
	\begin{equation}\label{eq:rescK}
	K_\ep(x):=\frac{1}{\ep^d}K\left(\frac{x}{\ep}\right)
	\end{equation}
and, for a measurable set $E\subset\Rd$ and $x\in\partial E$,
we define
	\begin{equation}\label{eq:resccurv}
	H_\ep(E,x):= \frac{1}{\ep} H_{K_\ep}(E,x).
	\end{equation}
We remark that
this scaling is mass preserving,
in the sense that, at least formally, $\norm{K}_{L^1(\Rd)}=\norm{K_\ep}_{L^1(\Rd)}$.
At the same time,
we  expect a localization effect in the limit. 
 
 Our main assumptions on the kernel $K$ are listed in Section \ref{sec:K}.
 In particular, we will  require that $K$ 
 is sufficiently regular and has  at most  a singularity in the origin, that is $ K\in W^{1,1}(\Rd\setminus B(0,r))$ for all $r>0$. 
 In addition, we assume that there exist $m>0$ and $s\in (0,1)$ such that
 	\[
 	0\leq K(x)\leq \frac{m}{\vass{x}^{d+1+s}} \quad \text{if } x\in B(0,1)^c,
 	\]
and that for all $\lambda>0$ and all $e\in \mathbb{S}^{d-1}:=\partial B(0,1)$ there holds
	\[
		K, \left|x\right|\left|\nabla K(x)\right|
		\in L^1\left(
			\left\{ x\in \Rd : \left| x \cdot e\right|
				\leq \frac{\lambda}{2} \vass{\pi_{e^\perp}(x)}^2 \right\}
		\right),
	\]
where   $e^\perp$ is  the hyperplane of vectors that are orthogonal to $e$,
and $\pi_{e^\perp}$ is the orthogonal projection operator on $e^\perp$.  
Actually,
in order to exploit these properties in our proofs,
we will need to make them  quantitative.
We refer the reader to Section \ref{sec:K}
for a detailed presentation of the assumptions. 
 
We point out that in \cite{dfm:convergenceof}
a similar problem was studied,
but there the assumptions on the interaction kernel,
and thus the choice of the rescaling,
are different from ours.
Indeed, the authors of \cite{dfm:convergenceof}
assume the kernel $K$ to be bounded near the origin
(hence nonsingular)
and to decay as $\left|x\right|^{-(d+1)}$ at infinity. 
The rescaled curvature is defined as
	\[
		\frac{1}{\ep\log \ep} H_{K_\ep}(E,x),
	\] 
and the authors prove that, as $\ep\to 0$,
it converges to an  anisotropic, local curvature functional.
They also show that
the rescaled  geometric motion approaches
the flow driven by the limiting curvature. 

In the last years,
other results related
to the asymptotic behavior of rescaled nonlocal functionals
have appeared in the literature,
mainly in the stationary setting.  
For radial, nonsingular kernels,
it is proved in \cite{mrt:nonlocalperimeter} that
the rescaled perimeters $\ep^{-1} \Per_{K_\ep}(E)$ converge pointwise
to the local perimeter functional.
In the same paper,
pointwise convergence of the rescaled curvature to the local mean curvature is obtained as well.
An improvement
concerning the convergence of perimeters
has recently been obtained in \cite{bp:onthe,pag:halfspaces},
where $\Gamma$-convergence of the functionals
$\ep^{-1} \Per_{K_\ep}(E)$ to De Giorgi's perimeter
is established for a class of singular kernels. Results in the same spirit addressing specifically the fractional case can be found in \cite{a,b,v}, see also \cite{sv} for $\Gamma$-convergence of nonlocal phase transitions. 
Finally, we recall the recent preprint \cite{cdnp:convergence}, where stability results for nonlocal geometric evolutions are studied by using 
viscosity solutions arguments. In the present paper, we propose a different, more geometric, approach to the problem, as we will detail in the following.

Our first main result is the uniform convergence
of the rescaled curvature functionals
to a local, anisotropic mean curvature functional,
when they are computed for smooth, compact sets.
We fix some notations needed to formulate the precise statement.

As before,  $p^\perp$ is the hyperplane
of the vectors that are orthogonal to $p$, and
$\pi_{p^\perp}$ is the orthogonal projection operator on $p^\perp$.
We denote by $\mathrm{Sym}(d)$
the space of $d\times d$ real symmetric matrices
and by $\mathcal{H}^{d-1}$
the $(d-1)$-dimensional Hausdorff measure.
For a $C^2$ hypersurface in $\Rd$ $\Sigma$,
we define the anisotropic mean curvature functional 
\begin{equation}\label{eq:H0intro}
	H_0(\Sigma,x) := -\frac{1}{\vass{\nabla\phi(x)}}
	\mathrm{tr}\left(
	M_K\left(\hat n\right)\pi_{\hat n^\perp} \nabla^2 \phi(x) \pi_{\hat n^\perp}
	\right),
	\end{equation}
where $\phi\in C^2(\Rd)$ is a function such that
$\Sigma \cap U = \{y\in\Rd :\phi(y) = 0\}\cap U$
in some open neighbourhood $U$ of $x$,
$\nabla\phi(x)\neq 0$,
$\hat{n}$ is the outer unit normal to $\Sigma$ at $x$, and finally 
	\begin{equation}\label{eq:MK}
	\begin{matrix}
	M_K\colon & \mathbb{S}^{d-1} & \longrightarrow & \mathrm{Sym}(d) \\
	& e & \longmapsto & \displaystyle{\int_{e^\perp}K(z)z\otimes z\de \mathcal{H}^{d-1}(z)}.
	\end{matrix}	
	\end{equation}
Then, we show the following: 

	\begin{theorem}\label{stm:main1}
		Let $K$ satisfy all  the assumptions in Section \ref{sec:K}. 
		Let $E\subset\Rd$ be a set
		whose boundary $\Sigma$ is compact and of class $C^2$. Then,
		\[
			\lim_{\ep\to 0^+}H_\ep(E,x)
				= H_0(\Sigma,x)
			\quad\text{uniformly in $x\in\Sigma$}.\]
	\end{theorem}
We recall that analog results to ours for nonsingular kernels are found in \cite{mrt:nonlocalperimeter} and in \cite{dfm:convergenceof},
respectively for the isotropic and the anisotropic case.

Our  second main result deals with
the convergence of the rescaled nonlocal geometric flows 
	\begin{equation}\label{int1}
	\partial_t x(t)\cdot \hat{n} = -H_\ep(E(t), x(t))
	\end{equation} 
to the  anisotropic mean curvature flow
	\begin{equation}\label{int2}
	\partial_t x(t)\cdot \hat{n} = -H_0(\Sigma(t), x(t)),
	\end{equation} 
where $\Sigma(t):=\partial E(t)$. 
We develop our analysis
in the framework of the level-set method.
This amounts to 
defining the evolving set $E(t)$ and its boundary $\Sigma(t)$
as the $0$ superlevel set and $0$ level set of some function $\phi(t,\,\cdot\,)$, which turns out to be a viscosity solution of 
the nonlocal  parabolic partial differential equation 
	\begin{equation}\label{eq:nlc-eq}
	\partial_t  \phi(t,x) + 
 \vass{\nabla \phi(t,x)} H_\ep (\{y:\phi(t,y)\geq \phi (t,x)\},x) = 0
	\end{equation}
if $E(t)$ solves the rescaled nonlocal geometric flow \eqref{int1},
or of the local   parabolic partial differential equation 
	\begin{equation}\label{eq:loc-eq}
	\partial_t  \phi(t,x) + 
	\vass{\nabla \phi(t,x)}H_0(\{y:\phi(t,y)= \phi (t,x)\},x) = 0
	\end{equation}
if $\Sigma(t)$ solve the anisotropic mean curvature flow \eqref{int2}. 
We can state our second major result. 

	\begin{theorem}\label{stm:main2}
		Let $K$ satisfy all  the assumptions in Section \ref{sec:K}. 
		Let $u_0\colon \Rd \to \R$ be a Lipschitz continuous function
		that is  constant outside a compact set.
		Let  $u_\ep, u\colon  [0,+\infty)\times\Rd \to \R$ be respectively
		the unique continuous viscosity solution to
		\eqref{eq:nlc-eq} and \eqref{eq:loc-eq},
		 with  initial datum $u_0$.
		Then
			\[\lim_{\ep\to 0} u_\ep(t,x)=u(t,x) \quad\text{locally uniformly in } [0,+\infty)\times\Rd.\] 
	\end{theorem}

The proof of Theorem \ref{stm:main2} is based on the convergence of curvatures obtained in Theorem \ref{stm:main1}. 
We propose a proof based on the concept of geometric barrier,
introduced by De Giorgi in \cite{deg:barriers}
as a weak solution to a wide range of evolution problems.
The study of barriers in relation to geometric parabolic PDEs,
such as \eqref{eq:loc-eq},
was developed by Bellettini, Novaga, and Paolini in the late 90's
\cite{bp:someresults,bel:alcunirisultati,bn:someaspects,bn:comparisonresults}.
It turns out that,
for the class of problems under consideration,
viscosity theory and barriers can be compared,
and this is the key point that
we will exploit in our analysis.
 
 We remark that isotropic fractional kernels  such as $K(y-x)=|y-x|^{-d-s}$ for $s\in (0,1)$ are not directly included
in the class of kernels we are considering, see Example \ref{fractional}.  Nevertheless the same kind of result as Theorem \ref{stm:main1} for the fractional mean curvature as $s\to 1$ was obtained in \cite{av,v, cdnp:convergence}, whereas  the convergence of the level set flow has been proved  in \cite{cdnp:convergence} by using viscosity solution methods.

Finally, we recall that there is a large literature concerning approximation results for mean curvature motions, either with local or nonlocal operators.
One of the most renowned algoritheorems is the threshold dynamics type one introduced in  \cite{BMO:diffusiongenerated}
by  Bence, Merriman, and Osher.
This approach was rigorously settled
in \cite{bg:asimple} and  \cite{eva:convergenceof};
then, the analysis was extended
to more general diffusion operators
in \cite{ish:ageneralizationof},
\cite{ips:thresholddynamics}, and \cite{cn:convergenceof}
(for anisotropic and crystalline evolutions).  
In \cite{cs:convergenceof} Caffarelli and Souganidis established
the convergence of an analogous  threshold dynamics scheme
to the (isotropic) motion by fractional
mean curvature,
and this result was adapted to  the anisotropic case, also in presence of a driving force,
in \cite{cnr:convex}. 

\subsection*{Structure of the paper}
In Section 2
we describe the class of interaction kernels
that we consider in this work.
In Section 3 and 4
we discuss some basic properties
of the curvatures functionals,
and we recall the level-set formulation for geometric flows,
the notion of geometric barriers,
and the main results about them. 
Section 5 is devoted to the proof of Theorem \ref{stm:main1}.
In Section 6, we provide a compactness result 
for the family of solutions to the rescaled nonlocal problems.
Eventually, Section 7 contains the proof of Theorem \ref{stm:main2}.

\subsection*{Acknowledgement} The authors warmly thank Matteo Novaga for inspiring discussions on this problem.

\section{Standing assumptions on the kernel}\label{sec:K}
Throughout this work,
$K\colon \Rd\setminus\{0\} \to [0,+\infty)$ is a measurable function
such that
\begin{equation}\label{eq:K-even}
K(y) = K(-y) \quad\text{for all } y\in\Rd\setminus\{0\}
\end{equation}
and
\begin{equation}\label{eq:W11}
K\in W^{1,1}(B(0,r)^c) \quad\text{for all } r>0.
\end{equation}
Note that \eqref{eq:W11} allows both $K$ and $\nabla K$
to be singular around the origin,
and it implies convergence of their integrals at infinity;
however, 
we need to make these information quantitative.

Firstly, we require that
	\begin{equation}\label{eq:sing-orig}
	\lim_{r\to 0^+} r\int_{B(0,r)^c} K(y)\de y = 0.
	\end{equation}
Then,  
for any  $e \in \mathbb{S}^{d-1}$ and $\lambda>0$,
we set
	\[Q_\lambda(e):=
		\left\{y\in \Rd : \vass{y\cdot e}\leq \frac{\lambda}{2} \vass{\pi_{e^\perp}(y)}^2\right\},\]
and we assume  that 		
	\begin{equation}\label{eq:imbert2}
	y\mapsto K(y),
	y\mapsto \vass{y} \vass{\nabla K(y)} 
	\in L^1(Q_\lambda(e))\quad \text{for all } e \in \mathbb{S}^{d-1} \text{ and } \lambda>0.
	\end{equation}
This will  imply that
sets with $C^{1,1}$ compact boundary have finite curvature,
see Proposition \ref{stm:imbert}.
We stress that we make no isotropy hypothesis on $K$;
still, we have to suppose some control
on the mass of $K$ in $Q_\lambda(e)$,
uniformly in $e$.
We therefore suppose that
for all $\lambda>0$ there exists $a_\lambda>0$
such that for all $e \in \mathbb{S}^{d-1}$
	\begin{equation}\label{eq:massa-parabole}
	\int_{Q_\lambda(e)} K(y)\de y \leq a_\lambda.
	\end{equation}
In addition, we require that 
there exist $a_0,b_0>0$ such that for all $e \in \mathbb{S}^{d-1}$
	\begin{eqnarray}
	\limsup_{\lambda\to 0^+}
		\frac{1}{\lambda}\int_{Q_\lambda(e)} K(y)\de y \leq a_0, \label{eq:imbert3} \\
	\limsup_{\lambda\to 0^+}
		\frac{1}{\lambda}\int_{Q_\lambda(e)} \vass{\nabla K(y)}\vass{y}\de y \leq b_0 \label{eq:imbert4}.
\end{eqnarray}
We assume as well that for all $e \in \mathbb{S}^{d-1}$
	\begin{equation}\label{eq:imbert5}
		\lim_{\lambda\to+\infty} \frac{1}{\lambda}\int_{Q_\lambda(e)} K(y)\de y = 0.
	\end{equation}	

Finally, we suppose that,
far from the origin,
$K$ is bounded above by a fractional  kernel;
that is, there exist $m>0$  and $s\in(0,1)$ such that
	\begin{equation}\label{eq:frac-decay}
		K(y)\leq \frac{m}{\vass{y}^{d+1+s}} \quad \text{if } y\in B(0,1)^c.
	\end{equation}

\begin{remark}
Inequality \eqref{eq:frac-decay} entails that
for all $\alpha < s$
	\begin{equation}\label{eq:imbert1}
	\lim_{r\to +\infty} r^{1+\alpha}\int_{B(0,r)^c} K(y)\de y = 0.
	\end{equation}
Actually, most of the results in the paper are not affected
if the weaker assumption \eqref{eq:imbert1} replaces  \eqref{eq:frac-decay}.
However, for the sake of simplicity,
we decided not to pursue this direction. 
\end{remark}


As a concluding comment about our assumptions on $K$,
we describe a class of singular kernels
that fits in our analysis.

\begin{example}[Fractional  kernels]\label{fractional} 
Let us suppose that $K\colon \Rd\setminus\{0\} \to [0,+\infty)$ satisfies
\eqref{eq:K-even} and that 
there exist constants $m,\mu>0$ and $s,\sigma\in(0,1)$ such that 
	\[K(y), \vass{y} \vass{\nabla K(y)} \leq \frac{\mu}{\vass{y}^{d+\sigma}}
		\quad \text{for all } y\in B(0,1)
	\]
and
	\[K(y), \vass{y} \vass{\nabla K(y)} \leq \frac{m}{\vass{y}^{d+1+s}}
		\quad \text{for all } y\in B(0,1)^c.\]
 Then, all the assumptions above are satisfied. 
 
Also fractional kernels with exponential decay at infinity fit in our framework;
namely, these are the kernels
$K\colon \Rd\setminus\{0\} \to [0,+\infty)$
that satisfy \eqref{eq:K-even} and
for which there exist constants
$m, \mu>0$ and $s\in(0,1)$ such that 
	\[
		K(y),  \vass{y} \vass{\nabla K(y)}
			\leq \frac{\mu e^{-m |y|}}{\vass{y}^{d+s}},\quad  \forall y\in\Rd.\]
\end{example} 

\section{Preliminaries about curvature functionals}
In this section we discuss some basic results
about the local and nonlocal curvature functionals $H_0$ and  $H_K$
defined in \eqref{eq:H0intro} and \eqref{eq:dfncurv}.

First of all,
we show that the nonlocal curvature is finite
on sets with $C^{1,1}$ boundaries.
Similar results are already available
in \cite{imb:levelset} and \cite{cmp:nonlocalcurvature}.
Nonetheless, 
we detail the argument for the sake of completeness,
and to recover estimate \eqref{eq:est-reg-hyp},
which will come in handy later.
We will use the following notation: 
for $e\in\mathbb{S}^{d-1}$, $x\in\Rd$ and $\delta>0$,
we denote the cylinder of center $x$ and axis $e$ as 
	\begin{equation}\label{cilindro} C_e(x,\delta):=
		\{y\in\Rd : y = x+z+te,\text{ with } z\in e^\perp\cap B(0,\delta),t\in(-\delta,\delta)\}.\end{equation}

\begin{proposition}\label{stm:imbert}
Let $E\subset \Rd$ be an open set
such that $\partial E$ is a $C^{1,1}$-hypersurface.
Then, for all $x\in\partial E$
there exist $\bar{\delta},\lambda>0$ such that 
	\begin{equation}\label{eq:est-reg-hyp}
	\vass{H_K(E,x)}\leq \int_{Q_{\lambda,\bar{\delta}}(\hat n)} K(y)\de y
	+ \int_{B(0,\bar\delta)^c}K(y)\de y,	
	\end{equation}
where $Q_{\lambda,\bar\delta}(\hat n):=\
\{y\in Q_\lambda(\hat n) : \vass{\pi_{\hat{n}^\perp}(y)}<\bar{\delta}\}$.
In particular, $H_K(E,x)$ is finite.
\end{proposition}
\begin{proof}
Let $\Sigma:= \partial E$ and 
$\hat n$ be  the outer unit normal to $\Sigma$ at $x$.
By the regularity of $\Sigma$, there exist
$\bar{\delta}:= \bar{\delta}(x)$ and
a function $f\colon \hat{n}^\perp\cap B(0,\bar{\delta})\to (-\bar{\delta},\bar{\delta})$
of class $C^{1,1}$ such that  
	\begin{gather}
	\Sigma\cap C_{\hat n}(x,\bar\delta)
		= \{y=x+z-f(z)\hat n: z\in \hat n^\perp \cap B(0,\bar\delta)\},
		\label{eq:graph} \\
	E\cap C_{\hat n}(x,\bar\delta) 
		= \{y=x+z-t\hat n: z\in \hat n^\perp \cap B(0,\bar\delta), t\in(f(z),\bar{\delta})\}, 
		\label{eq:epigraph} \\
  \vass{f(z)}\leq \dfrac{\lambda}{2} \vass{z}^2\quad\text{for some } \lambda>0.	 \label{ftre}
	\end{gather}
It is not restrictive to assume $r<\bar\delta$;
hence, we can split the integral in \eqref{eq:dfncurv}
into the sum
\[\int_C K(y-x)\tilde\chi_E(y)\chi_{B(x,r)^c}(y) \de y
+ \int_{C^c}K(y-x)\tilde\chi_E(y)\de y,\]
where we set $C:= C_{\hat n}(x,\bar\delta)$. 
The second term above  is finite as a consequence of \eqref{eq:W11};
indeed, since $B(x,\bar\delta)\subset C$, we have that
\begin{equation}\label{eq:stima-lontano}
\vass{\int_{C^c}K(y-x)\tilde\chi_E(y)\de y} \leq \int_{B(0,\bar\delta)^c}K(y)\de y.
\end{equation}

So, we are  left to show that the integral
\[I_r := \int_C K(y-x)\tilde\chi_E(y)\chi_{B(x,r)^c}(y) \de y\]
is bounded by a constant that does not depend on $r$.	
Taking into account \eqref{eq:epigraph} and
recalling that $K$ belongs to $L^1(B(0,r)^c)$ for any $r>0$,
we can write
\[	I_r	 = \int_{\hat n^\perp\cap B(0,\bar\delta)}\left[
\int_{f(z)}^{\bar\delta} K(z-t\hat n) b_r(z,t) \de t
-\int_{-\bar\delta}^{f(z)} K(z-t\hat n) b_r(z,t) \de t
\right]\de\mathcal{H}^{d-1}(z),
\]
where, for $(z,t)\in [\hat n^\perp\cap B(0,\bar\delta)]\times (-\bar{\delta},\bar{\delta})$, 
\begin{equation}\label{eq:br}
b_r(z,t):= \begin{cases}
0 & \text{if } \vass{z}< r\text{ and } \vass{t} < \sqrt{r^2-\vass{z}^2}\\
1 & \text{otherwise}
\end{cases}.
\end{equation}
Since $K$ is even,
we get 
\[\begin{split}
I_r = & \int_{\hat n^\perp\cap B(0,\bar\delta)}\left[
\int_{f(z)}^{\bar\delta} K(z-t\hat n) b_r(z,t) \de t
- \int_{-f(-z)}^{\bar\delta} K(z-t\hat n) b_r(z,t) \de t
\right]\de\mathcal{H}^{d-1}(z) \\
= & -\int_{\hat n^\perp\cap B(0,\bar\delta)}
\int_{-f(-z)}^{f(z)} K(z-t\hat n) b_r(z,t) \de t \de\mathcal{H}^{d-1}(z)
\end{split}\] 
In view of \eqref{ftre} we infer
\[\begin{split}
\vass{I_r}	\leq &
\int_{\hat n^\perp\cap B(0,\bar\delta)}
\int_{-\frac{\lambda}{2}\vass{z}^2}^{\frac{\lambda}{2}\vass{z}^2} K(z-t\hat n) b_r(z,t) \de t \de\mathcal{H}^{d-1}(z) \\
= &
\int_{Q_{\lambda,\bar{\delta}}(\hat n)} K(y)\chi_{B(0,r)^c}(y)\de y.
\end{split}\]
Assumption \eqref{eq:imbert2} allows to take the limit in the last inequality,
and we  conclude that \eqref{eq:est-reg-hyp} holds.
\end{proof}

\begin{remark}
	We point out that
	\eqref{eq:est-reg-hyp} has been obtained
	just exploiting the facts that
	$K$ is even,  $K\in L^1(B(0,r)^c)$ for all $r>0$,
	and that $K\in L^1(Q_\lambda(e))$
	for all $e\in \mathbb{S}^{d-1}$  and $\lambda>0$.
	
	We next observe that
	in \eqref{eq:est-reg-hyp}  the second integral takes into account
	the ``tails'' of the kernel $K$,
	while the first one is related to the second fundamental form of $\Sigma$.
	We will prove in the sequel that,
	under our standing assumptions,
	the second  term is negligible  in the large scale limit. 
\end{remark}

The next lemma collects two fundamental properties of $H_K$.
We omit the proofs, which can derived easily from the definition of $H_K$.

\begin{lemma}\label{stm:CMP}
	Let $E\subset\Rd$ be an open set
	such that $H_K(E,x)$ is finite for some $x\in\partial E$.
	\begin{enumerate} 
		\item For any $h\in \Rd$ and any orthogonal matrix $R$,
			if $T(y):= Ry + h$, then
			\begin{equation}\label{eq:nl-cng-var}
			H_K(E,x)=H_{\tilde K}(T(E),T(x)),
			\end{equation}
			where $\tilde K := K\circ R^{\mathsf{t}}$.
			In particular, $H_K$ is invariant under translation.
		\item If $F\subset E$ and
			$x\in\partial E \cap\partial F$,
			then $H_K(E,x)\leq H_K(F,x)$.
	\end{enumerate}
\end{lemma}
%

We focus now on the functional  $H_0$ defined in \eqref{eq:H0intro},
which is a local anisotropic mean curvature functional,
the anisotropy being encoded by $M_K$.
As a first step, we establish the well-posedness of $M_K$ and
to this aim we recall the characterization of Sobolev functions
in terms of absolute continuity on lines,
whose definition we include here: 

\begin{definition} Let $\Omega\subset\Rd$ be an open set. 
	A function $u\colon \Omega\to\R$ is absolutely continuous on lines 
	if $u$ is Borel measurable in $\Omega$ and
	locally absolutely continuous on almost all lines parallel to coordinate axes,
	that is, if $\{e_1,\dots,e_d\}$ is the canonical basis,
	for all $i=1,\dots,d$ there exists $N_i\subset e_i^\perp$
	such that $\mathcal{H}^{d-1}(N_i)=0$
	and for all $z\in e_i^\perp \cap N_i^c$
	the function $I\ni t\mapsto u(z+te_i)$ is absolutely continuous
	on any compact interval $I$
	such that $z+te_i\in \Omega$ when $t\in I$. 
	\end{definition} 
 
Since absolutely continuous functions are differentiable a.e.,
we highlight that
if $u$ is absolutely continuous on lines,
then it admits partial  derivatives  a.e. and hence the  vector  $\nabla u$ is a.e. defined.
On the other hand, if a function has Sobolev regularity,
then it has a representative which is absolutely continuous on lines.
That is the content of the following result,
whose proof can be found in \cite[Theorem 2.3]{hai}.

\begin{theorem}\label{stm:acl}
	Let $\Omega\subset\Rd$ be an open set.
	For any $p\in[1,+\infty)$,
	$u\colon \Omega\to \R$ belongs to the Sobolev space $W^{1,p}(\Omega)$
	if and only it coincides a.e. with a function $\tilde u\in L^p(\Omega)$
	that is absolutely continuous on lines
	and whose gradient $\nabla \tilde u$ belongs to $L^p(\Omega;\Rd)$.
\end{theorem}

Thanks to \eqref{eq:W11} and to the theorem above,
we may without loss of generality suppose that
the kernel $K$ is absolutely continuous on lines in $B(0,r)^c$ for all $r>0$. 
We exploit this fact to prove boundedness and continuity of $M_K$. 

\begin{lemma}\label{stm:MK}
	Let $a_0$ be the constant in \eqref{eq:imbert3}.
	Then, for all $e\in\mathbb{S}^{d-1}$,
		\begin{equation}\label{eq:Kz2}
		\int_{e^\perp}K(z)\vass{z}^2\de\mathcal{H}^{d-1}(z)\leq a_0,
		\end{equation}
	and $M_K$ is continuous on $\mathbb{S}^{d-1}$.
	
	Moreover, for any $e\in\mathbb{S}^{d-1}$,
	there holds
	\begin{equation}\label{eq:decad-Kz2}
	\lim_{r\to+\infty} r^{\beta}\int_{e^\perp \cap B(0,r)^c} K(z)\vass{z}^2\de\mathcal{H}^{d-1}(z)=0
	\quad \text{for all } \beta<s.
	\end{equation}
\end{lemma}
\begin{proof}
	By (a slight adaptation of) Theorem \ref{stm:acl},	
	for any $e\in\mathbb{S}^{d-1}$ and any $j\in\N$,
	there exists a $\mathcal{H}^{d-1}$-negligible
	$N_j \subset \{z\in e^\perp : j\vass{z}\geq1\}$
	such that, for all $z\in e^\perp\cap N_j^c$ with $j\vass{z}\geq1$,
	the function $t\mapsto K(z+te)$ is absolutely continuous
 	when $t$  belongs to closed, bounded intervals. 
	By the arbitrariness of $j\in\N$,
	we conclude that
	for $\mathcal{H}^{d-1}$-a.e. $z\in e^\perp$,
	$[a,b] \ni t\mapsto K(z+te)$ is absolutely continuous for any $a,b\in\R$.
	
	Hence, by the Mean Value Theorem,
	for $\mathcal{H}^{d-1}$-almost every $z\in e^\perp$
	we find
	\begin{equation}\label{eq:mean-val}
	\lim_{\lambda\to0^+} \frac{1}{\lambda}
	\int_{-\frac{\lambda}{2}\vass{z}^2}^{\frac{\lambda}{2}\vass{z}^2}
	K(z+te)\de t = K(z)\vass{z}^2.
	\end{equation}
	Now, for any $\lambda>0$,
	\eqref{eq:imbert2} guarantees that
	\[a_\lambda(e):= \int_{Q_\lambda(e)} K(y)\de y\in(0, +\infty).\]
	Moreover, we have
	\[
	\frac{1}{\lambda}\int_{e^\perp}
	\int_{-\frac{\lambda}{2}\vass{z}^2}^{\frac{\lambda}{2}\vass{z}^2}
	K(z+te)\de t\de\mathcal{H}^{d-1}(z) =
	\frac{a_\lambda(e)}{\lambda}.			
	\]
	In view of \eqref{eq:mean-val} and \eqref{eq:imbert3},
	we can take the limit $\lambda\to 0^+$
	on both sides of the last equality
	and this yields \eqref{eq:Kz2},
	as desired.
	
	Now we prove that $M_K$ is continuous.
	We fix $e\in\mathbb{S}^{d-1}$
	and we consider a sequence of rotations $R_n$
	such that $R_n\to \id$.
	We have
		\[\begin{split}
			\vass{M_K(R_n e)-M_K(e)} = &\,
				\vass{\int_{e^\perp} K(R_nz)\, R_nz\otimes R_nz\de\mathcal{H}^{d-1}
				-\int_{e^\perp} K(z)z\otimes z\de\mathcal{H}^{d-1}} \\
			\leq &\, \vass{\int_{e^\perp} K(R_nz)\, [R_nz\otimes R_nz-z\otimes z]\de\mathcal{H}^{d-1}} \\
				&\,+\vass{\int_{e^\perp} [K(R_nz)-K(z)]z\otimes z\de\mathcal{H}^{d-1}}.	
		\end{split}\]
	Since $K\in L^1(B(0,r)^c)$ for all $r>0$, it holds
	$$\lim_{n\to+\infty}\norm{K\circ R_n - K}_{L^1(B(0,r)^c)}=0;$$
	hence, we deduce that $K(R_nz)\to K(z)$ for $\mathcal{H}^{d-1}$-a.e. $z\in e^\perp$
	and	this, together with \eqref{eq:Kz2},
	gets that the upper bound we have on $\vass{M_K(R_n e)-M_K(e)}$
	vanishes as $n\to+\infty$.
	
Estimate \eqref{eq:decad-Kz2} is an easy consequence of assumption \eqref{eq:frac-decay}. \end{proof}

From the very definition of $M_K$,
we notice that
$\pi_{\hat n^\perp}M_K\left(\hat n\right)\pi_{\hat n^\perp} = M_K\left(\hat{n}\right)$. Using this, we observe that if $\Sigma,x,\phi$, and $\hat{n}$ are the same as in \eqref{eq:H0intro},
we have
	\begin{align}\label{eq:H0}
	H_0(\Sigma,x) & = -\dfrac{1}{\vass{\nabla\phi(x)}}\mathrm{tr}\left(
						M_K\left(\hat n\right)\nabla^2 \phi(x)
					\right)   \\
				& = - 
					\dfrac{1}{\vass{\nabla\phi(x)}}\int_{\hat n^\perp}
						K(z)\nabla^2 \phi(x)z\cdot 	z\de\mathcal{H}^{d-1}(z).  \nonumber
	\end{align}
	\begin{remark}		
		Let us consider a smooth hypersurface $\Sigma$
		whose outer unit normal at a given point $x$ is $\hat n$,
		and the map $T(y):= Ry + h$,
		where $R$ is an orthogonal matrix and $h\in\Rd$.
		Then, it is easy to check by using  \eqref{eq:H0}
		that it  holds
			\begin{equation}\label{eq:l-cng-var}
			H_0(\Sigma,x)=\tilde H_0(T(\Sigma),T(x)),
			\end{equation}
		where $\tilde H_0$ is the anisotropic mean curvature functional
		associated with the kernel $\tilde K:= K \circ R^\mathrm{t}$.
		To prove our claim, we observe that
		if $\Sigma = \{y\in\Rd : \phi(y)=0\}$
		for some smooth $\phi\colon \Rd\to\R$,
		then $T(\Sigma)=\{y\in\Rd : \psi(y)=0\}$
		with $\psi(y):= \phi( R^\mathsf{t}(y-x))$.
		We have
			\[\nabla\psi(T(y))=R\nabla\phi(y)
			\quad\text{and}\quad
			\nabla^2\psi(T(y))=R\nabla^2\phi(y)R^\mathsf{t},
			\]
		and, therefore,
			\[\begin{split}
				\tilde H_0(T(\Sigma),T(x)) = &
					-\frac{1}{\vass{R\nabla\phi(x)}}\int_{R(\hat n^\perp)}
					\tilde K(z)\left(R\nabla^2\phi(x)R^\mathsf{t}\right)z\cdot z\de\mathcal{H}^{d-1}(z) \\
					= & -\frac{1}{\vass{\nabla\phi(x)}}\int_{\hat n^\perp}
					K(z)\nabla^2\phi(x)z\cdot z\de\mathcal{H}^{d-1}(z).
			\end{split}\]
	\end{remark}

\begin{remark}[Connection with standard mean curvature] \label{remarkradial} 
When $K$ is  radial,
that is, $K(x)=K_0(\vass{x})$ for some $K_0\colon (0,+\infty)\to [0,+\infty)$,
then $H_0$ coincides with the standard mean curvature,
up to a multiplicative constant.	
Indeed, let $\Sigma$ be a $C^2$ hypersurface
	such that $0\in\Sigma$ and
	$\Sigma\cap U = \{y\in U : \phi(y)=0\}$
	for some  neighbourhood $U$ of $0$ and
	some smooth function $\phi\colon U\to \R$.
	We suppose also that $\nabla\phi(0)\neq 0$ and
	that the outer unit normal to $\Sigma$ at $0$ is $e_d$.
	We recall the expression
	of the  mean curvature $H$ of $\Sigma$ at $0$:
		\[ H(\Sigma,0) =
			-\frac{1}{\omega_{d-1}\vass{\nabla\phi(0)}}
			\int_{\mathbb{S}^{d-2}} \nabla^2\phi(0)e\cdot e\, \de\mathcal{H}^{d-2}(e),
		\]
	with $\omega_{d-1} := \mathcal{H}^{d-1}(\mathbb{S}^{d-1})$.

If  $K(x)=K_0(\vass{x})$, then formula \eqref{eq:Kz2} reads
		\[ c_K:= \int_0^{+\infty} r^d K_0(r)\de r<+\infty,\]
			and, consequently, we have
		\begin{align*}
			H_0(\Sigma,0) &= -
					\frac{1}{\vass{\nabla\phi(0)}}
				\int_0^{+\infty} r^d K_0(r)\de r
					\int_{\mathbb{S}^{d-2}}\nabla^2\phi(0)e\cdot e\,\de\mathcal{H}^{d-2}(e) \\
						&=  \omega_{d-1}\,c_K H(\Sigma,0).
 		\end{align*}
\end{remark}

\section{Barriers and level-set flow for geometric evolutions}
We devote this section to some basics about 
level-set formulations and barriers for the geometric flows driven by the curvatures $H_K$ and $H_0$.
In particular, we recall existence and uniqueness results for the level-set flow,
and  we revise its connections with the notion of geometric barriers. 

We consider the following  geometric evolutions 
for the family of sets $\{E(t)\}_{t\geq 0}$:
	\begin{equation}\label{gf} \partial_t x(t)\cdot \hat{n} = - H_\ep(E(t), x),\qquad \partial_t x(t)\cdot \hat{n} = -H_0(E(t), x),\end{equation}
where $\hat{n}$ is the outer unit normal 
to $\partial E(t)$ at the point $x(t)$ and $H_\ep\ep$ is the rescaled version of $H_K$ defined in \eqref{eq:resccurv}.
In addition, we accompany these equations with an initial datum $E_0$,
which we assume to be a bounded set. 	
 
Let us begin with the level-set formulations of the geometric flows \eqref{gf}.
First of all, we interpret the initial datum $E_0$ as the superlevel set of a suitable function $u_0:= \Rd\to\R$.
Explicitly, we suppose that $E_0=\{x : u_0(x)\geq 0\}$ and $\partial E_0=\{x : u_0(x)=0\}$;
moreover, throughout the paper we assume that 
	\begin{equation}\label{initialdatum} 
	u_0\colon \Rd \to\R
	\text{ is Lipschitz  and constant outside a compact $C$.} 
	\end{equation}
Then, we consider the nonlocal and local Cauchy problems: 
	\begin{eqnarray}\label{eq:eps_pb}
	\begin{cases}
	\partial_t  u(t,x) + 
	 \vass{\nabla u(t,x)} H_\ep(\{y:u(t,y)\geq u (t,x)\},x) = 0
		\\
	\hspace{7.5cm} (t,x)\in[0,+\infty)\times\Rd \\
	u(0,x) = u_0(x) \hspace{7.4cm}  x\in\Rd
	\end{cases},
\\  \label{eq:limit_pb}
	 	\begin{cases}
	 	\partial_t u(t,x)
	 		- \mathrm{tr}\left(M_K\left(\widehat{\nabla u(t,x)}\right)\nabla^2 u(t,x)\right)
	 		= 0
	 						& (t,x)\in[0,\infty)\times\Rd \\
	 	u(0,x) = u_0(x) 	& x\in\Rd
	 \end{cases}.
	 \end{eqnarray}
Observe that
	\[\vass{\nabla u(x)}H_0(\{y:u(y)= u (x)\},x)
		=	-\mathrm{tr}\left(M_K\left(\widehat{\nabla u(x)}\right)\nabla^2 u(x)\right)\] 
(recall that $\hat{p}:= p/\vass{p}$ if $p\neq 0$).
	
We  remind the  definition of viscosity solution
for nonlocal equations,
which goes back to the work \cite{sle:approximationschemes},
see also 
\cite{imb:levelset,dfm:convergenceof,cmp:nonlocalcurvature,cdn:fatteningand}.

\begin{definition}[Solution to the rescaled problems]
	A locally bounded, upper semicontinuous function  (resp. lower semicontinuous) 
	$u_\ep: [0,+\infty)\times \Rd \to\R$  is a viscosity subsolution (resp. supersolution)
	to the problem \eqref{eq:eps_pb} if
	\begin{enumerate}
		\item $u_\ep(0,x)\leq u_0(x)$ for all $x\in\Rd$ (resp. $u_\ep(0,x)\geq u_0(x)$);
		\item for all $(t,x)\in(0,+\infty)\times \Rd$
		and for all $\phi\in C^2([0,+\infty)\times \Rd)$ such that
		$u_\ep -\phi$ has a maximum at $(t,x)$ (resp. has a minimum at $(t,x)$), 
		it holds
		\[\partial_t\phi(t,x)\leq 0 \quad  (\text{resp. }\partial_t\phi(t,x)\geq 0)\qquad
		\text{when } \nabla\phi(t,x)=0, 
		\]
		or 
		\begin{gather*}
		\partial_t \phi(t,x)
		+  \vass{\phi(t,x)} H_\ep(\{y:\phi(t,y)\geq \phi (t,x)\},x) \leq 0 \\
		(\text{resp. } 
		\partial_t \phi(t,x)
		+  \vass{\phi(t,x)} H_\ep(\{y:\phi(t,y) > \phi (t,x)\},x)\geq 0)
			\quad\text{otherwise.}
		\end{gather*}
	\end{enumerate}

	A continuous function $u_\ep: [0,+\infty)\times \Rd \to\R$
	is a viscosity solution to \eqref{eq:eps_pb}
	if it is both a viscosity sub- and supersolution. 
\end{definition}

Existence and uniqueness 
of a viscosity solution to \eqref{eq:eps_pb}
were proved in \cite{cmp:nonlocalcurvature}, 
in a very general setting.
A similar result can also be found in \cite{imb:levelset}.

\begin{theorem}[Comparison principle and existence of solutions to the nonlocal problem]\label{esistenzaeps}
If the standing assumptions on the kernel and \eqref{initialdatum} hold,
for all $\ep>0$,
	if $v_\ep, w_\ep: [0,+\infty)\times\Rd \to \R$
	are respectively a sub- and a supersolution to \eqref{eq:eps_pb},
	then $v_\ep(t,x)\leq w_\ep(t,x)$ for all $(t,x)\in [0,+\infty)\times\Rd$.
	
	Moreover, \eqref{eq:eps_pb} admits
	a unique bounded, Lipschitz continuous
	viscosity solution in $[0,+\infty)\times\Rd$,
	which is constant in $\Rd\setminus C$, for some compact set $C\subset \Rd$.  
\end{theorem}  
 
We recall also the definition of solution to the limit problem \eqref{eq:limit_pb}, see \cite{dfm:convergenceof}.
	
	\begin{definition}[Solution to the limit problem]
	A locally bounded, upper semicontinuous function (resp. lower semicontinuous function) 
	$u: [0,\infty)\times \Rd\to\R$ is a viscosity subsolution (resp. supersolution) 
	to the Cauchy's problem \eqref{eq:limit_pb} if
		\begin{enumerate}
		\item $u(0,x)\leq u_0(x)$ for all $x\in\Rd$, (resp. $u(0,x)\geq u_0(x)$);
		\item for all $(t,x)\in(0,+\infty)\times \Rd$
			and for all $\phi\in C^2([0,+\infty)\times \Rd)$ such that
			$u -\phi$ has a maximum at $(t,x)$ (resp. a minimum at $(t,x)$)
			it holds
				\[\partial_t\phi(t,x)\leq 0\quad(\text{resp. }\partial_t\phi(t,x)\geq 0)
					\qquad \text{when } \nabla\phi(t,x)=0 \text{ and } \nabla^2\phi(t,x)=0
			\]
			or 
				\[\partial_t \phi(t,x)
				- \mathrm{tr}\left(M_K\left(\widehat{\nabla \phi(t,x)}\right)\nabla^2 \phi(t,x)\right) \leq 0 \   (\text{resp. }\geq 0)
					\quad\text{otherwise.}
			\]
		\end{enumerate}
	A continuous function $u: [0,+\infty)\times \Rd \to\R$
	is a viscosity solution to \eqref{eq:limit_pb}
	if it is both a viscosity sub- and supersolution.
	\end{definition}
 
As for existence of solutions,
we observe that the function
	\begin{equation*}
	\begin{matrix}
		F_0: & \Rd\setminus\{0\}\times \mathrm{Sym}(d)&  \longrightarrow & \R \\
					& (p,X) & \longmapsto &
			\displaystyle{-\mathrm{tr}\left(M_K\left(\hat{p}\right)X\right)}
	\end{matrix}	
	\end{equation*}
that defines the problem \eqref{eq:limit_pb}
has the three following properties:
	\begin{enumerate}
		\item it is continuous;
		\item it is geometric, that is,
			for all $\lambda >0$, $\sigma\in\R$, $p\in\Rd\setminus\{0\}$
			and $X\in\mathrm{Sym}(d)$ it holds
			$F_0(\lambda p, \lambda X + \sigma p\otimes p)=\lambda F_0(p,X)$.
		\item it is degenerate elliptic, that is,
			$F_0(p,X)\geq F_0(p,Y)$
			for all $p\in\Rd\setminus\{0\}$
			and $X,Y\in\mathrm{Sym}(d)$ such that $X\leq Y$.
	\end{enumerate}
It is well known \cite{bn:comparisonresults,cgg:uniquenessand}
that these conditions grant
existence and uniqueness of a viscosity solution:

	\begin{theorem}\label{esistenza0}
		Let us suppose that \eqref{initialdatum} holds.
		Then, the Cauchy's problem \eqref{eq:limit_pb} admits
		a unique bounded, Lipschitz continuous  viscosity solution in $[0,+\infty)\times \Rd$,
		which is constant in $\Rd\setminus C$, for some compact set $C\subset \Rd$.
	\end{theorem}

Summing up,
owing to Theorems \ref{esistenzaeps} and \ref{esistenza0},
we get that, for every initial datum $u_0$ as in \eqref{initialdatum},
there exist a unique viscosity solution $u_\ep$ to \eqref{eq:limit_pb} and
a unique viscosity solution $u$ to \eqref{eq:limit_pb}. 
We define the level-set flows associated with these solutions. For every $\lambda\in \R$, we set 
 \begin{eqnarray}\label{levelseteps} E_{\ep, \lambda}^+(t)=& \{x\in\Rd : u_\ep(t,x)\geq \lambda\}, \qquad & E_{\ep, \lambda}^-(t)=\{x\in\Rd : u_\ep(t,x)>\lambda\}, \\
 E^{+}_{\lambda}(t)= & \{x\in\Rd : u(t,x)\geq \lambda\}, \qquad & E^{-}_{\lambda}(t)=  \{x\in\Rd : u(t,x)> \lambda\}.
\label{levelset0}  
 \end{eqnarray}
 
It is well known that, as long as they are smooth,
these families are solutions to the geometric flows \eqref{gf} resp. with $H_\ep$ and $H_0$ and initial datum $E_\lambda=\{x\in\Rd : u_0(x)\geq \lambda\}$. 
 
Geometric evolutions may be formulated
as PDEs involving distance functions from the moving front,
see for instance the survey \cite{amb:geometricevolution} by Ambrosio;
in the following definitions,
we use them to express a regularity property
both in time and space for a class of evolving sets
(see \ref{stm:ii} below) w.r.t. a generic geometric law.

\begin{definition}\label{defistrict}
Let $0\leq t_0<t_1<+\infty$. 
We say that the evolutions of sets
$[t_0,t_1]\ni t\mapsto D(t)$ is a
\emph{geometric subsolution (resp. supersolution)}
to the flow associated with the curvature functional $H$ if
	\begin{enumerate}
		\item\label{stm:i} $D(t)$ is closed and  $\partial D(t)$ is compact
			for all $t\in[t_0,t_1]$;
		\item\label{stm:ii} there exists an open set $U\subset \Rd$
			such that the distance function $(t,x)\mapsto \de_{D(t)}(x)$
			is of class $C^\infty$ in $[t_0,t_1]\times U$
			and $\partial D(t)\subset U$ for all $t\in[t_0,t_1]$;
		\item for all $t\in(t_0,t_1)$ and $x(t)\in\partial D(t)$,
			it holds
				\begin{equation}\label{eq:smooth-sub}
				\partial_t x(t)\cdot \hat{n} \leq -H( D(t),x(t))\qquad(\text{resp. }\partial_t x(t)\cdot \hat{n} \geq -H( D(t),x(t)),
				\end{equation}
			where $\hat{n} $ is the outer unit normal to $D(t)$ at $x$.
	\end{enumerate}

When strict inequalities hold,
$D(t)$  is called \emph{strict geometric subsolution (resp. strict geometric supersolution)}.
\end{definition}
 
\begin{remark}
We notice that, for any $p\in \Rd\setminus\{0\}$ and
$X\in\mathrm{Sym}(d)$, by \eqref{eq:Kz2} we get that 
	\[\vass{\mathrm{tr}(M_K(\hat{p})X)} = 
		\frac{1}{2}\vass{\int_{\hat{p}^\perp}
			K(z)\, Xz\cdot z\de\mathcal{H}^{d-1}(z)}
		\leq \frac{a_0}{2}\vass{X},
	\]
This ensures that
geometric sub- and supersolution for the flow associated with $H_0$ exist
(see \cite{bn:comparisonresults}).
\end{remark}

Next, we remind the notion of geometric barriers
w.r.t. these smooth evolutions:

	\begin{definition}\label{defbarrier}
	Let $T>0$ and  $\mathcal{F}^-$ and $\mathcal{F}^+$
	be, respectively, the families of
strict geometric sub- and supersolution 
	to the flow  associated with some curvature functional $H$, as introduced in Definition \ref{defistrict}. 
		\begin{enumerate}
			\item We say that the evolution of sets
			$[0,T]\ni t\mapsto E(t)$ is an \emph{outer barrier}
			w.r.t. $\mathcal{F}^-$ (resp. $\mathcal{F}^+$)
			if whenever $[t_0,t_1]\subset[0,T]$
			and $[t_0,t_1]\ni t\mapsto D(t)$ is a
			smooth strict subsolution
			(resp. $F(t)$ is a smooth strict  supersolution) such that
			$D(t_0)\subset E(t_0)$,
			then  we get  $D(t_1)\subset E(t_1)$ (resp. such that $F(t_0)\subset E(t_0)$,
			then  we get  $F(t_1)\subset E(t_1)$).
			\item Analogously,
			$[0,T]\ni t\mapsto E(t)$ is an \emph{inner barrier}
			w.r.t. the family $\mathcal{F}^-$ (resp. $\mathcal{F}^+$)
			if whenever $[t_0,t_1]\subset[0,T]$
			and $[t_0,t_1]\ni t\mapsto D(t)$ is a
			smooth strict subsolution			
			(resp. supersolution) such that
			$E(t_0) \subset \mathrm{int}(D(t_0))$,
			then $E(t_1) \subset \mathrm{int}(D(t_1))$ (resp. $E(t_0) \subset \mathrm{int}(F(t_0))$,
			then $E(t_1) \subset \mathrm{int}(F(t_1))$).			
		\end{enumerate}
	 \end{definition}

We are interested in barriers
for the anisotropic mean curvature motion \eqref{gf} 
because they are comparable with level-sets flows,
as the next theorem shows.
Its proof can be found in \cite[theorem 3.2]{bn:comparisonresults}.
For further reading about barriers
for general geometric, local evolution problems,
we refer to that paper and to \cite{bn:someaspects}. 
	\begin{theorem}\label{stm:BN}
		Let  $u$ be the unique solution to \eqref{eq:limit_pb} with initial datum $u_0$ as in \eqref{initialdatum}.
		Let $E^{\pm}_{\lambda}$ the sets defined in \eqref{levelset0}.  
				\begin{enumerate}
			\item The map 
				$[0,T]\ni t\mapsto E^-_\lambda(t)$
				is the minimal outer barrier for
				the family of strict geometric subsolutions
				associated with $H_0$,
				that is $E^-_\lambda(t)$ is an outer barrier
				and $E^-_\lambda(t)\subset E(t)$
				for any other outer barrier $E(t)$.
			\item The map $[0,T]\ni t\mapsto E^+_\lambda(t)$ 
				is the maximal inner barrier for
				the family of  geometric  strict supersolutions
				associated with $H_0$,
				that is $E^+_\lambda(t)$ is an inner barrier
				and $E(t) \subset E^+_\lambda(t)$
				for any other inner barrier $E(t)$.
		\end{enumerate}
	\end{theorem}

Lastly, we  mention a comparison principle
concerning the level-set flow and  strict geometric sub- and supersolutions for the nonlocal problems,
see \cite[Proposition A.10]{cdn:fatteningand}.
	\begin{proposition}\label{stm:comp-CDNV}
		Let $u_\ep: [0,+\infty)\times \Rd\to \R$
		be the viscosity solution to \eqref{eq:eps_pb} with initial datum $u_0$ as in \eqref{initialdatum}. 
		Let $E_{\ep,\lambda}^{\pm}(t)$ 
		be as in \eqref{levelseteps}.
		Then, the evolutions
		$t\mapsto E_{\ep,\lambda}^-(t)$
		and $t\mapsto E_{\ep,\lambda}^+(t)$
		are, respectively,
		an outer barrier w.r.t  geometric  strict subsolutions to \eqref{eq:eps_pb}
		and an inner barrier w.r.t geometric  strict supersolutions to \eqref{eq:eps_pb}.
	\end{proposition}

\section{Convergence of the rescaled nonlocal curvatures}
This section is devoted to the proof of Theorem \ref{stm:main1},
the first main result of the paper.
The argument consists of two steps:
firstly, we deal in Lemma \ref{stm:conv-rate} with
the pointwise convergence of the curvatures,
providing a precise estimate on the error;
then, in Proposition \ref{stm:unif-conv}, we show that
it is possible to make the estimate uniform
when smooth, compact hypersurfaces are considered.

We fix the notations that
we are going to use in the current section.
Let $E\subset \Rd$  be a set of class $C^2$. 
Then for all $x\in \Sigma := \partial E$, 
there exist an open neighborhood $U$ of $x$
and  $\phi\in C^2(U)$
such that
	\begin{equation*}
	\Sigma\cap U = \{y \in U : \phi(y)=0\},
	\quad
	E \cap U =\{y \in U : \phi(y)>0\},
	\end{equation*}
and $\nabla\phi(y)\neq 0$ for all $y\in\Sigma\cap U$.
We write $\hat{n}$
for the outer unit normal to $\Sigma$ at $x$.
Lastly, by the Implicit Function Theorem,
there exist
$\bar\delta:= \bar{\delta}(x)>0$ and 
$f\colon \hat{n}^\perp\cap B(0,\bar\delta) \to (-\bar{\delta},\bar{\delta})$
such that \eqref{eq:graph} and \eqref{eq:epigraph} hold,
and $\inf_{y\in C_{\hat{n}}(x,\bar\delta)}\vass{\nabla\phi(y)}>0$.

\begin{lemma}\label{stm:conv-rate}
Let $E\subseteq \Rd$ be such that $\Sigma:=\partial E$ is of class  $C^2$.
Let $x\in \Sigma$, $\bar{\delta}$, and $f$ be as above, and
let $s\in (0,1)$ be the exponent in \eqref{eq:frac-decay}.
Then, for all $\alpha,\beta\in(0,s)$,
there exist	$q>1$ and $\bar{\ep}\in(0,1)$
such that $q\bar{\ep}\leq\bar{\delta}$
and that, for all $\ep\in(0,\bar{\ep})$ and all $\delta \in(q\ep,\bar{\delta})$,
it holds
		\[
		\vass{ H_\ep(E,x) - H_0(\Sigma,x)} \leq \mathcal{E}(\ep,\delta),
		\]
where			
		\begin{equation}\label{eq:conv-rate} 
			\mathcal{E}(\ep,\delta):=  \,
				\frac{1}{\delta}\left(\frac{\ep}{\delta}\right)^{\alpha}
				+ (b_0+1)\norm{\nabla^2f}^2_{L^\infty(D)}\delta
				+ a_0\,\omega_f(\delta)
				+ \vass{\nabla^2 f(0)} \left(\frac{\ep}{\delta}\right)^\beta,
		\end{equation}
with $D :=  \hat n^\perp\cap B(0,\bar\delta)$
and
	\begin{equation}\label{eq:omegaf}
		\omega_f(\delta):=  \sup_{z\in B(0,\delta)}\vass{\nabla^2 f(z)-\nabla^2 f(0)}.
	\end{equation}
\end{lemma}
\begin{proof}
We start by observing that,
without loss of generality,
we may assume that $x=0$ and
$\hat{n}=e_d:= (0,\dots,0,1)$. 
 
 The argument is similar to the one
followed to prove estimate \eqref{eq:est-reg-hyp}.
There exists $f\colon D \to (-\bar{\delta},\bar{\delta})$ of class $C^2$
such that $f(0)=0$, $\nabla f(0)=0$, and \eqref{eq:graph} and \eqref{eq:epigraph} hold.  Moreover,
		\begin{gather}
		\partial_i f = \dfrac{\partial_i\phi}{\partial_d \phi}, \\
		\partial^2_{i,j}f= \dfrac{1}{\partial_d \phi}
			\left(\partial^2_{i.j}\phi + \partial_i f\,\partial^2_{j,d}\phi
			+ \partial_j f\,\partial^2_{i,d}\phi + \partial_i f\,\partial_j f\,\partial^2_{d,d}\phi\right) 
			\label{eq:implicit2}
		\end{gather}
for $i,j=1,\dots,d-1$.
Let us introduce  the function
	\[
		f_\ep(z):= \frac{f(\ep z)}{\ep}.
	\]
Since  $f$ is of class $C^2$,
for all $z\in D$ there exists $z'$ such that 
$f_\ep(z)=(\ep \nabla^2 f(\ep z')z\cdot z) / 2$. 	When $t$ ranges between $-f_\ep(-z)$ and $f_\ep(z)$, we thus see that
	\begin{equation}\label{eq:stima-quadr}
	\left| t \right| \leq \frac{\ep}{2}\norm{\nabla^2 f}_{L^\infty(D)}\vass{z}^2.
	\end{equation}

Let us fix  $0<\ep<\delta<\bar{\delta}$.
We split $H_\ep$ into two different contributions:
	\[
		H_\ep(E, 0)
		= I_{\ep}^0+I_\ep^1:=  -\frac{1}{\ep}\int_{C} K_\ep(y)\tilde\chi_E(y)\de y
		- \frac{1}{\ep}\int_{C^c} K_\ep(y)\tilde\chi_E(y)\de y,
	\]
where  $C:= C_{e_d}(0,\delta)$.
The first integral  takes into 	account the interactions
with points that are close to $0$, 
and it approximates the anisotropic mean curvature at $0$
when $\ep$ is small;
the second term encodes the energy stored far away from the origin.
Observe that
	\begin{equation*}
	I^0_\ep =
		\frac{1}{\ep}\int_{e_d^\perp\cap B\left(0,\frac{\delta}{\ep}\right)}
		\int_{-f_\ep(-z)}^{f_\ep(z)} K(z+te_d)\de t\de\mathcal{H}^{d-1}(z).
	\end{equation*}
Let us define
	\[
		J_\ep :=   \frac{1}{\ep}\int_{e_d^\perp\cap B\left(0,\frac{\delta}{\ep}\right)}
				K(z) \left[f_\ep(z)+f_\ep(-z)\right]\de\mathcal{H}^{d-1}(z), 
					\]
and recall that, in view of \eqref{eq:H0},
	\[
		H_0(\Sigma, 0) =  \int_{e_d^\perp} K(z)\nabla^2 f(0)z\cdot z\de\mathcal{H}^{d-1}(z).
	\]	
We consider the chain of inequalities
	\[
		\left| H_\ep(E,0)-H_0(\Sigma,0) \right|
			= \left| I_\ep^0 + I_\ep^1 - H_0(\Sigma, 0) \right|
			\leq \vass{I^0_\ep-J_\ep} + \vass{J_\ep-H_0(\Sigma,0)} + \vass{I^1_\ep},
		\]
and we estimate each term separately. 

We start with $I_\ep^1$.
We remark that,
as a consequence of \eqref{eq:imbert1},
for all $\alpha<s$ there exists $q_1>1$ such that
	\begin{equation}\label{eq:Ieps1}
			\vass{I_\ep^1}= \frac{1}{\ep}\int_{B\left(0,\frac{\delta}{\ep}\right)^c}K(y)\de y\leq \frac{1}{\delta}\left(\frac{\ep}{\delta}\right)^{\alpha}
			\quad\text{whenever } q_1\ep <\delta.
	\end{equation}

We proceed with the other terms.
We observe that 
		\begin{equation}\label{eq:Ieps0Jeps}
		\vass{I^0_\ep-J_\ep} \leq
			\frac{1}{\ep}\int_{e_d^\perp\cap B\left(0,\frac{\delta}{\ep}\right)}
				\vass{\int_{-f_\ep(-z)}^{f_\ep(z)} \left[K(z+te_d)-K(z)\right]\de t}\de\mathcal{H}^{d-1}(z).	
		\end{equation}
By Theorem \ref{stm:acl},
for $\mathcal{H}^{d-1}$-a.e. $z\in e_d^\perp$,
it holds
	\[
		K(z+te_d)-K(z)=\int_0^t \partial_d K(z+se_d) \de s,
	\]
and this,
combined with \eqref{eq:stima-quadr},
implies that
	\begin{equation*}
			\vass{K(z+te_d)-K(z)} \leq
				\int_{-\frac{\ep}{2}\norm{\nabla^2f}_{L^\infty(D)}\vass{z}^2}^{\frac{\ep}{2}\norm{\nabla^2f}_{L^\infty(D)}\vass{z}^2}
					\vass{\nabla K(z+se_d)}\de s.
		\end{equation*}
We plug this inequality in \eqref{eq:Ieps0Jeps} and
we obtain
		\begin{equation*}
		\begin{split}
		& \vass{I^0_\ep-J_\ep} \\ 
		&\qquad \leq\norm{\nabla^2f}_{L^\infty(D)}
				\int_{e_d^\perp\cap B\left(0,\frac{\delta}{\ep}\right)} \vass{z}^2
				\int_{-\frac{\ep}{2}\norm{\nabla^2f}_{L^\infty(D)}\vass{z}^2}^{\frac{\ep}{2}\norm{\nabla^2f}_{L^\infty(D)}\vass{z}^2}
			\vass{\nabla K(z+se_d)}\de s \de\mathcal{H}^{d-1}(z)\\
		&\qquad \leq \norm{\nabla^2f}_{L^\infty(D)}\frac{\delta}{\ep}
				\int_{Q(\ep)}\vass{y}\vass{\nabla K(y)}\de y,
		\end{split}
		\end{equation*}
where  $Q(\ep):=  Q_{\ep\norm{\nabla^2f}_{L^\infty(D)}}(e_d)$. 		
By using \eqref{eq:imbert4}
we get that 
there exists $\eta\in(0,\bar{\delta})$ such that  
	\begin{equation}\label{eq:Ieps0Jeps2}
		\vass{I^0_\ep-J_\ep}\leq (b_0+1)\norm{\nabla^2f}^2_{L^\infty(D)}\delta
		\quad\text{whenever } \ep <\eta.
	\end{equation}
	
Finally, we have		
		\begin{multline*}
			\vass{J_\ep-H_0(\Sigma, 0)} \leq \omega_f(\delta) \int_{e_d^\perp\cap B\left(0,\frac{\delta}{\ep}\right)}
					K(z)\vass{z}^2\de\mathcal{H}^{d-1}(z)
					\\ + \vass{\nabla^2 f(0)}
					\int_{e_d^\perp\cap B\left(0,\frac{\delta}{\ep}\right)^c}
						K(z)\vass{z}^2\de\mathcal{H}^{d-1}(z),
		\end{multline*}
$\omega_f$ being defined in \eqref{eq:omegaf}.
Thanks to \eqref{eq:decad-Kz2},
	for all $\beta<s$, there exists $q_2>0$
	such that, if $q_2\ep <\delta$, then
		\[\left(\frac{\delta}{\ep}\right)^\beta
			\int_{e_d^\perp\cap B\left(0,\frac{\delta}{\ep}\right)^c}
				K(z)\vass{z}^2\de\mathcal{H}^{d-1}(z) \leq 1.
		\]
Recalling \eqref{eq:Kz2}, we thus find
		\begin{equation}\label{eq:JepsJ0}
			\vass{J_\ep-H_0(\Sigma,0)} \leq a_0\,\omega_f(\delta)
				+ \vass{\nabla^2 f(0)} \left(\frac{\ep}{\delta}\right)^\beta
			\quad\text{whenever } q_2\ep <\delta.
		\end{equation}
	
	Now, if we set $q:=  \max\{q_1,q_2\}>1$
	with $q_1$ and $q_2$ as above,
	both \eqref{eq:Ieps1} and \eqref{eq:JepsJ0}
	hold for all $\ep,\delta>0$ such that
	$q\ep < \delta<\bar{\delta}$.
	Besides, if we pick
	$\bar{\ep}:= \min\left\{\eta,\bar{\delta}/ q \right\}$,
	\eqref{eq:Ieps0Jeps2} is satisfied as well
	whenever $\ep<\bar{\ep}$.
	This yields the conclusion. 
\end{proof}

	\begin{remark}
		In the proof of Lemma \ref{stm:conv-rate},
		we  did not exploit assumptions 
	  \eqref{eq:sing-orig} and \eqref{eq:imbert5}. These will be useful in the proof of  
		Proposition \ref{aprioriest}.
	\end{remark}
	
By applying the estimate on the error term given in Lemma \ref{stm:conv-rate},
we deduce the desired uniform convergence. 

\begin{proposition}\label{stm:unif-conv}
Under the same notation and assumptions of  Lemma \ref{stm:conv-rate},
there exists a constant $c:=  c(\alpha,\beta,a_0,b_0)>0$ 
	such that for all $\gamma\in \left(0, \frac{\alpha}{1+\alpha}\right)$,
	it holds
		\begin{multline*} 
		\vass{H_\ep(E,x) - H_0(\Sigma,x)} \leq \\
			c\Bigl(\ep^{\alpha-\gamma(1+\alpha)}
				+ \norm{\nabla^2f}_{L^\infty(D)}\ep^\gamma  
				 +\omega_f(q\ep^\gamma)
				+ \vass{\nabla^2 f(0)} \ep^{(1-\gamma)\beta}\Bigr).
		\end{multline*}
	In particular, if $\Sigma$ is compact,
	the conclusion of Theorem \ref{stm:main1} holds.
\end{proposition}
\begin{proof}
	We start by proving that
	pointwise convergence holds.
	We choose $\gamma\in\left(0, \alpha/(1+\alpha)\right)$ 
	and we observe that, for any $\ep<\bar{\ep}< 1$,
	we have	$q\ep < q\ep^\gamma$.
	We may therefore pick $\delta = q\ep^\gamma$ in \eqref{eq:conv-rate}
	and check that $\mathcal{E}(\ep, q\ep^\gamma)\to 0$
	when $\ep\to 0^+$. 
	The pointwise convergence follows.
	
	Now, we turn to the case
	when $\Sigma$ is compact and of class $C^2$.
	We denote by $\hat{n}_x$ the outer unit normal to $\Sigma$ at $x$
	and by $\hat{n}_x^\perp$ the tangent plane at the same point.
	Let us also define
\[		 		V_\Sigma(\delta):=   \{y\in\Rd : \inf_{z\in\Sigma} |y-z|<\delta\}, \]
	and
		\[\bar\delta :=  \sup\{\delta>0 :
				\text{ the boundary of $V_\Sigma(\delta)$ is of class $C^2$}\}>0.\]
	This ensures that,
	for any $x\in\Sigma$, the implicit function $f$
	defined on $\hat{n}_x^\perp$ 
	ranges in $(-\bar\delta,\bar\delta)$. Let us denote this function by $f_x$
	to stress that it  depends on $x$. 
	There exists $\bar{\ep}<1$ such that
	for all $\ep\in(0,\bar{\ep})$,
	for all $\gamma\in \left(0, \alpha/(1+\alpha)\right)$,
	and	for all $x\in\Sigma$ it holds
		\begin{multline*}
		\vass{ H_\ep(E,x) - H_0(\Sigma,x)}\\ \leq
		c\Bigl(\ep^{\alpha-\gamma(1+\alpha)}
		+ \norm{\nabla^2f_x}_{L^\infty(\hat{n}_x^\perp\cap B(0,\bar{\delta}))}\ep^\gamma 
		  +\omega_{f_x}(q\ep^\gamma)
		+ \vass{\nabla^2 f_x(0)} \ep^{(1-\gamma)\beta}\Bigr).
		\end{multline*}
	Since $\Sigma$ is compact,
	$\vass{\nabla^2 f_x(0)}$ and
	$\norm{\nabla^2 f_x}_{L^\infty(\hat{n}_x^\perp\cap B(0,\bar{\delta}))}$
	are bounded above by the $L^\infty(\Sigma)$-norm
	of the second fundamental form of $\Sigma$;	
also,  there exists a function $\omega_\Sigma$ 
	that vanishes in $0$, that is decreasing
	and that satisfies $\omega_{f_x}(\delta)\leq \omega_\Sigma(\delta)$
	whenever $\delta$ is sufficiently small.
	In conclusion, we obtain an estimate on
	$\vass{  H_\ep(E,x)-H_0(\Sigma,x)}$
	that is uniform in $x$, and the thesis holds.
\end{proof}

\section{A priori estimates for the  rescaled problems} 
In this section we establish a compactness property
for  the family of solutions
to the Cauchy's problems \eqref{eq:eps_pb}.  
Even though the result is known, we sketch its proof,
because it is not explicitly stated in the literature for our setting. 
	\begin{proposition}\label{aprioriest}
		Assume that $u_0\colon \Rd \to\R$ is as in \eqref{initialdatum}, 
		and let $u_\ep$ be the unique continuous viscosity solution to \eqref{eq:eps_pb}. 
		Then, 
			\begin{equation}\label{equilip}
				\vass{u_\ep(t,x)-u_\ep(t,y)}\leq
					\norm{\nabla u_0}_{L^\infty(\Rd)} \vass{x-y}
				\quad \text{for all } t\in [0,T]
					\text{ and } x,y\in \Rd,
			\end{equation}
		and there exists a constant $c>0$ independent of $\ep$ such that
			\begin{equation}\label{equihold}
				\vass{u_\ep(t,x)-u_\ep(s,x)}
					\leq \norm{\nabla u_0}_{L^\infty(\Rd)}\sqrt{c\vass{t-s}}
				\quad \text{for all } t,s\in [0,T]
					\text{ and } x\in \Rd.
		\end{equation}
	\end{proposition}		
	\begin{proof}
	The equi-Lipschitz property \eqref{equilip}
	is a consequence
	of the Lipschitz continuity of the datum and of the comparison principle.
	We skip the proof,
	since it is completely standard and 
	can be found, for instance, in \cite{cdn:fatteningand,dfm:convergenceof}.
	
	For the proof of equi-H\"older continuity, we follow the strategy of Section 5 in \cite{dfm:convergenceof}.
	We point out that, however,
	the case that we treat differs from the one in the reference, 
	mainly because of the possible singularity of our interaction kernel.
	
	We fix $\eta>0$ and $x\in\Rd$
	and we consider
		\begin{equation}\label{test}
			\phi(t,y)= Lt + A\sqrt{\vass{y-x}^2+\eta^2}+u_0(x),
		\end{equation}
	where $A:=:=\norm{\nabla u_0}_{L^\infty(\Rd)}$. 
	We claim that, for $L>0$ sufficiently large,
	$\phi$ is a supersolution to \eqref{eq:eps_pb} for any $\ep\in(0,1)$.
	
	To prove the claim, we remark first of all
	that $\phi(0,y)\geq u_0(y)$
	as a consequence of the Lipschitz continuity of $u_0$.
	Also, we observe that, for any $y\in \Rd$,
		\[\{z\in\Rd : \phi(t,z)\geq \phi(t,y)\}
			= B(x,\vass{y-x})^c.\]
	Hence, to show that $\phi$ is a supersolution,
	it is sufficent to choose $L$ so large that 
		\[\frac{L}{A} \geq \frac{ \vass{y-x}}{\sqrt{\vass{y-x}^2+\eta^2}}H_\ep(B(x,\vass{y-x}),y)
		\quad\text{for all }y\in\Rd \text{ and } \ep\in(0,1).\] 
	Recalling that the nonlocal curvature is invariant under translations,
	if we set $e :=  \widehat{y-x}$ and $r:=  \vass{y-x}$,
	we have that the last inequality holds if and only if
		\begin{equation}\label{eq:unif-supersol}
		\frac{L}{A} \geq \frac{r}{\sqrt{r^2+\eta^2}}H_\ep(B(-re,r),0)
		\quad\text{for all }r>0,e\in\mathbb{S}^{d-1} \text{ and } \ep\in(0,1).
		\end{equation}
So, we are left to prove that there exists 
	$L_0:=  L_0(\eta)>0$ such that
		\begin{equation}\label{sup}
			\sup_{ r>0,\,e\in\mathbb{S}^{d-1}}\sup_{\ep\in(0,1)}
			\frac{r}{ \sqrt{r^2+\eta^2}}H_\ep(B(-re,r),0)
			\leq L_0;
		\end{equation}
	this clearly yields \eqref{eq:unif-supersol} for $L=AL_0$.
	
	To recover estimate \eqref{sup},
	we use inequality \eqref{eq:est-reg-hyp}.
	We get
		\[0 \leq H_\ep(B(-re,r),0)
			\leq \int_{Q_{\frac{\ep}{r}}(e)} K(y)\de y
				+ \int_{B\left(0,\frac{r}{2\ep}\right)^c}K(y)\de y,\]
	and hence
		\begin{equation*}
			\frac{r}{\sqrt{r^2+\eta^2}}H_\ep(B(-re,r),0)
				\leq \frac{r}{\ep\eta}
				\left[\int_{Q_{\frac{\ep}{r}}(e)} K(y)\de y
				+ \int_{B\left(0,\frac{r}{2\ep}\right)^c}K(y)\de y\right].
		\end{equation*}
	By assumptions
	\eqref{eq:sing-orig}, \eqref{eq:imbert5},
	\eqref{eq:imbert3}, and \eqref{eq:imbert1},
	there exist $\lambda,\Lambda>0$ with the following properties:
		\begin{enumerate}
			\item $\lambda<\Lambda$;
			\item if $r<\lambda \ep$, then
					\[\frac{r}{\ep}\int_{Q_{\frac{\ep}{r}}(e)} K(y)\de y\leq \frac{1}{2}
					\quad\text{and}\quad
					\frac{r}{\ep}\int_{B\left(0,\frac{r}{2\ep}\right)^c}K(y)\de y \leq\frac{1}{2}\]
				and, consequently,
					\begin{equation}\label{eq:regime1}
					\frac{r}{\sqrt{r^2+\eta^2}}H_\ep(B(-re,r),0) \leq \frac{1}{\eta};
					\end{equation}
			\item if $r>\Lambda \ep$, then
					\[\frac{r}{\ep}\int_{Q_{\frac{\ep}{r}}(e)} K(y)\de y\leq a_0 + \frac{1}{2}
					\quad\text{and}\quad
					\frac{r}{\ep}\int_{B\left(0,\frac{r}{2\ep}\right)^c}K(y)\de y \leq\frac{1}{2}\]
				and, consequently,
					\begin{equation}\label{eq:regime2}
					\frac{r}{\sqrt{r^2+\eta^2}}H_\ep(B(-re,r),0) \leq \frac{a_0+1}{\eta}.
					\end{equation}
		\end{enumerate}
	Now, only the case
	$\lambda\ep \leq r\leq \Lambda\ep$ is left to discuss.
	In this intermediate regime, 
	recalling \eqref{eq:massa-parabole},
	we easily obtain
		\begin{equation}\label{eq:regime3}
		\frac{r}{\sqrt{r^2+\eta^2}}H_\ep(B(-re,r),0)
			\leq \frac{\Lambda}{\eta}
				\left(c + \int_{B\left(0,\frac{\lambda}{2}\right)^c}K(y)\de y\right),
		\end{equation}
	with $c>0$ depending only on $\lambda$.
	
In view of \eqref{eq:regime1}, \eqref{eq:regime2}, and \eqref{eq:regime3},
there exists a constant
	$c:= c(a_0,\lambda,\Lambda)>0$ such that
		\[\sup_{ r>0,\,e\in\mathbb{S}^{d-1}}\sup_{\ep\in(0,1)}
		\frac{r}{\sqrt{r^2+\eta^2}}H_\ep(B(-re,r),0)
		\leq \frac{c}{\eta},
		\]
	and \eqref{eq:unif-supersol} thus holds
	for the choice $ L = Ac/\eta$.
	
Summing up, we proved  that,
for any fixed $x\in\Rd$, the function
		\[\phi(t,y)=A\left(\frac{c}{\eta}t + \sqrt{\vass{y-x}^2+\eta^2} \right) + u_0(x)\]
	is a supersolution to \eqref{eq:eps_pb} for any $\ep>0$.

By means of an analogous argument
we can prove that,
for all $x\in\R^d$, the function
	\[
		\psi(y):=-A\left(\frac{c}{\eta}t + \sqrt{\vass{y-x}^2+\eta^2} \right) + u_0(x),
	\] 
is a subsolution to \eqref{eq:eps_pb} for any $\ep>0$ and some $c=c(a_0,\lambda,\Lambda)$.

All in all,
thanks to the comparison principle in Theorem \ref{esistenzaeps},
we infer that for all $(t,x)\in[0,T]\times\Rd$ and all $\eta>0$, 
		\[\vass{u_\ep(t,x)-u_0(x)}
			\leq \norm{\nabla u_0}_{L^\infty(\Rd)}\left(\frac{c}{\eta}t + \eta\right).\]
The previous estimates holds for every $\eta$, and hence, by choosing $\eta=\sqrt{ct}$, we get
		\begin{equation}\label{eq:q-equihold}
		\vass{u_\ep(t,x)-u_0(x)}\leq  2\norm{\nabla u_0}_{L^\infty(\Rd)}\sqrt{ct}.
		\end{equation}
Eventually, we deduce \eqref{equihold} from \eqref{eq:q-equihold}
	by combining the facts that
	the problem \eqref{eq:eps_pb} is invariant w.r.t. translations in time,
	that it admits a unique solution, and
	that $\norm{\nabla u_\ep(t,\,\cdot\,)}_{L^\infty(\Rd)}\leq \norm{\nabla u_0}_{L^\infty(\Rd)}$
	for all $t\in[0,T]$.
\end{proof}

\section{Convergence to the solution of the limit problem}
This section is devoted
to the proof of the second main result of the paper, Theorem \ref{stm:main2}. 
Theorem \ref{stm:main1} establishes an asymptotic link
between the rescaled nonlocal curvatures and
the anisotropic mean curvature.
In what follows,
we take advantage of this relationship
to deduce locally uniform convergence
of the viscosity solutions $u_\ep$ of \eqref{eq:eps_pb}
 to the viscosity solution $u$ of \eqref{eq:limit_pb}. 

To achieve the result,
we compare any limit point $v$ of $\{u_\ep\}$
(which Proposition \ref{aprioriest} proves to be a relatively compact family)
with the viscosity solution $u$ to \eqref{eq:limit_pb}. 
More precisely, we focus on the respective superlevel sets, and,
by using  the theory of geometric barriers and their relations with the level-set flows, 
we establish the inclusions \eqref{inclusion1} and \eqref{inclusion2}. 
In turn, these are sufficient to conclude that $v=u$,
thanks to the next lemma. 
	\begin{lemma}\label{stm:incl-suplev}
		Let $f,g:\Rd \to \R$ be  two continuous  functions such that  for all $\lambda \in \R$ there  hold 
			\[\{x\in\Rd : f(x)>\lambda\}\subseteq \{x\in\Rd : g(x)\geq \lambda\}\] and \[ 
		  \{x\in\Rd : g(x)>\lambda\}\subseteq \{x\in\Rd : f(x)\geq \lambda\}.\]
		Then, $f(x)=g(x)$ for all $x\in\Rd$.	
	\end{lemma}
	\begin{proof}
		Let $\bar x\in\Rd$ and
		assume that $g(\bar x)=\lambda$.
		Then,  for all $\mu>0$, we get $\bar x\in \{x : g(x)>\lambda-\mu\}\subseteq  \{x :f(x)\geq \lambda-\mu\}$, which in particular    implies $f(\bar x)\geq \lambda$.
			If    $f(\bar x)>\lambda$, then for some $\mu_0>0$, we would get 
		$\bar x\in \{x :f(x)>\lambda+\mu_0\}\subseteq  \{x : g(x)\geq \lambda+\mu_0>\lambda\}$, in contradiction with the fact that  $g(\bar x)=\lambda$. 
		So  $f(\bar x)=\lambda$.
		By reversing the role of $f$ and $g$,
		we get the conclusion. 
	\end{proof}

Let  $\lambda\in \R$ and  $E^{\pm}_{\ep,\lambda} (t)$ be the  level-set flows associated with the solutions $u_\ep$ to \eqref{eq:eps_pb} defined  in \eqref{levelseteps}.  We introduce the families
$\tilde E^{\pm}_\lambda (t)$,
which are the set-theoretic upper limits of $E_{\ep, \lambda}^\pm(t)$:
	\begin{equation}\label{eq:limsup}
	\tilde E^-_\lambda(t) := \bigcap_{\ep<1} \bigcup_{\eta<\ep} E^{-}_{\eta, \lambda}(t)
	\quad\text{and}\quad
	\tilde E^+_\lambda(t) := \bigcap_{\ep<1} \bigcup_{\eta<\ep} E^{+}_{\eta, \lambda}(t). 	
	\end{equation} 

	\begin{remark}\label{ind}
	It is an immediate consequence of the definition
	that, for any $\bar\ep<1$,
		\begin{equation*}
		\tilde E^-_\lambda(t) = \bigcap_{\ep<\bar\ep}
									\bigcup_{\eta<\ep} E^{-}_{\eta, \lambda}(t)
		\quad\text{and}\quad
		\tilde E^+_\lambda(t) = \bigcap_{\ep<\bar\ep}
									\bigcup_{\eta<\ep} E^{+}_{\eta, \lambda}(t). 	
		\end{equation*}
	\end{remark}

We are ready to discuss the proof
of our convergence result:

\begin{proof}[Proof of Theorem \ref{stm:main2}]
We divide the proof in three steps,
starting with a preliminary observation.
By Proposition \ref{aprioriest},
we know that the family $u_\ep$ is relatively compact
in $C([0,T]\times \Rd)$ and,
consequently, there exist a subsequence $\{u_{\ep_n}\}$
and a function $v\in C([0,T]\times \Rd)$
such that $u_{\ep_n}\to v$ locally uniformly as $\ep\to 0^+$.
We remark that
the conclusion is achieved if we show that $v=u$.
Indeed, since the argument applies
to any converging subsequence of $\{u_\ep\}$,
it follows that
the whole family $\{u_\ep\}$ locally uniformly converges to $u$, as desired.

From now on we reason on a subsequence
that we still denote $\{u_\ep\}$ and
that we suppose to be locally uniformly converging to $v$.

\smallskip
\noindent\textbf{Step 1:
		we claim that   for every $\lambda\in\R$, 
			\begin{equation}\label{inclusion1}
			\{x\in\Rd : v(t,x)>\lambda\} \subseteq 
			\tilde E^-_\lambda(t) \subseteq
			\tilde E^+_\lambda(t) \subseteq 
			\{x\in\Rd : v(t,x)\geq \lambda\}
			\end{equation}
			with $\tilde E^{\pm}_\lambda(t)$ as in \eqref{eq:limsup}.}
		
		In this part of the proof we exploit only the pointwise convergence of $\{u_\ep\}$.
		Without loss of generality,
		we discuss just the case $\lambda = 0$.
		
		Let us fix $\bar x\in \Rd$
		such that $v(t,\bar x) >0$,
		that is, $v(t,\bar x)=\mu$ for some $\mu>0$.
		Since $v$ is the limit of $\{u_\ep\}$,
		there exists $\bar\ep>0$ such that
			\[u_\ep(t,\bar x)\geq \frac{\mu}{2}>0 \quad \text{for all $\ep<\bar\ep$},\]
		and hence $\bar x\in  \tilde E^-_0(t)$.
		This shows that $\{x\in\Rd : v(t,x)>0 \}\subseteq \tilde E^-_0(t)$.
		
		Let us now turn to the inclusion
		$\tilde E^+_0(t)\subseteq \{x\in\Rd : v(t,x)\geq 0 \}$.
		By definition, if $\bar x\in \tilde E^+_0(t)$,
		then for all $\ep<1$ there exists $\eta_\ep<\ep$
		such that $u_{\eta_\ep}(t,\bar x)\geq 0$. 
		Taking the limit $\ep\to 0$,
		we get
			$$v(t, \bar x) = \lim_{\ep\to 0} u_{\eta_\ep}(t,\bar x) \geq 0.$$
		
\smallskip
\noindent\textbf{Step 2: we claim that, for all $\lambda\in \R$, 
			\begin{equation}\label{inclusion2}
				\{x\in\Rd : u(t,x)>\lambda\}\subseteq
				\tilde E^-_\lambda(t)\subseteq
				\tilde E^+_\lambda(t)\subseteq 
				\{x\in\Rd : u(t,x)\geq \lambda\}
			\end{equation}
where $u$ is the viscosity solution to \eqref{eq:limit_pb}}. 

We will firstly show that 
$ \tilde E_{\lambda}^-(t)$ and $\tilde E_{\lambda}^+(t)$ are, respectively,
an outer barrier for the family of strict geometric subsolutions
and an inner barrier  
associated with the flow of $H_0$. 
If these assertions hold true,
then Theorem \ref{stm:BN} immediately entails the conclusion,
because it states that
$\{x\in\Rd : u(t,x)>\lambda\}$ is the minimal outer  barrier
for the family of strict geometric subsolutions, and that $\{x\in\Rd : u(t,x)\geq \lambda\}$ is the maximal inner barrier
for the family of strict geometric supersolutions. 

We prove just that 
$ \tilde E_{0}^-(t)$ is an outer barrier  for
the family of strict geometric subsolutions, since 
the arguments for $\lambda\neq 0$ and $ \tilde E_{0}^+(t)$ are the same. 	
		 
Let us consider, for some $0\leq t_0<t_1\leq T$,
a family  of evolving sets 	$t\mapsto D(t)$  
which is  a strict geometric
		subsolution to the anisotropic mean curvature motion
		when $t\in [t_0, t_1]$.
		Explicitly, we suppose that
		there exists $\ell>0$ such that  
			\begin{equation}\label{eq:strict-geom}
			\partial_t x(t)\cdot \hat{n}_D(t,x(t)) \leq -H_0(\partial D(t),x(t)) -\ell
			\quad \text{for all } t\in (t_0,t_1]\text{ and }x(t)\in \partial D(t),
			\end{equation}
		where $\hat{n}_D$ is 
		the outer unit normal to $D(t)$;
		we assume as well that 
			\begin{equation}\label{eq:incl-tempo0}
			D(t_0) \subset  	\tilde E_{0}^-(t_0).	\end{equation}
		We want to show that $D(t_1) \subset  	\tilde E_{0}^-(t_1)$. 

Recalling definition \eqref{eq:limsup},
we get from \eqref{eq:incl-tempo0} that
for all $\ep<1$ there exists $\eta_\ep\leq \ep$ such that
\begin{equation}\label{ic} D(t_0)\subseteq E_{\eta_\ep, 0}^-(t_0).\end{equation} 
		Since for $t\in [t_0, t_1]$ the second fundamental forms of $\partial D(t)$ are uniformly bounded, we can apply   Theorem \ref{stm:main1} and we deduce that 
			\[
				\lim_{\ep\to 0}H_\ep(D(t),x)=H_0(D(t),x)
					\quad \text{uniformly in } t\in[t_0,t_1]\text{ and }x\in\partial D(t).
			\]
		Consequently,
		there exists $\bar\ep:= \bar\ep(\ell)$ such that,
		for all  $\ep<\bar\ep$,
			\begin{equation*}
			\partial_t x(t)\cdot \hat{n}_D(t,x(t)) \leq -H_\ep(D(t),x(t)) -\frac{\ell}{2}
			\quad \text{for all } t\in (t_0,t_1]\text{ and }x(t)\in \partial D(t),
			\end{equation*}
		or, in other words,  
		$t\mapsto D(t)$ 
		is a strict geometric subsolution to
		all the rescaled problems of parameter $\ep\in(0,\bar\ep)$.
		By  \eqref{ic} and Proposition \ref{stm:comp-CDNV}, we obtain that
		for all $\ep<\bar\ep$ there exists $\eta_\ep\leq \ep$ such that 
			\[D(t)\subset E_{\eta_\ep,0}^-(t)\quad \text{for all } t\in[t_0, t_1].\]
 		We take advantage of  Remark \ref{ind} to deduce from the previous inclusion that
 			\[
 			D(t)\subseteq \tilde E_{0}^-(t)\quad 
				  \text{for all } t\in[t_0,t_1].
			\]		
		In particular, we conclude that $D(t_1)\subseteq \tilde E_{0}^-(t_1)$, as desired. 			
	
\smallskip
\noindent\textbf{Step 3: we conclude $v=u$.}
 
By 	\eqref{inclusion1} and \eqref{inclusion2},
		we deduce that,	for every $\lambda\in\R$ and $t\in [0, T]$,
			\[
			\{x\in\Rd : v(t,x)>\lambda\} \subseteq \{x\in\Rd : u(t,x)\geq\lambda\},\]\[
			\{x\in\Rd : u(t,x)>\lambda\} \subseteq \{x\in\Rd : v(t,x)\geq\lambda\}.
			\]
		The proof is thus accomplished by applying Lemma \ref{stm:incl-suplev}. 
	\end{proof} 

\end{document}